\documentclass[12pt,twoside]{article}
\usepackage{amsmath,amssymb}
\usepackage{amsthm}
\usepackage{upref}
\usepackage{citeref}
\numberwithin{equation}{section}

\theoremstyle{plain}
\newtheorem{theorem}{Theorem}[section]
\newtheorem{proposition}[theorem]{Proposition}
\newtheorem{definition}[theorem]{Definition}

\theoremstyle{definition}
\newtheorem{remark}[theorem]{Remark}

\begin{document}

\newcommand{\eq}{equation}
\newcommand{\real}{\ensuremath{\mathbb R}}
\newcommand{\comp}{\ensuremath{\mathbb C}}
\newcommand{\rn}{\ensuremath{{\mathbb R}^n}}
\newcommand{\tn}{\ensuremath{{\mathbb T}^n}}
\newcommand{\rnp}{\ensuremath{\real^n_+}}
\newcommand{\rnn}{\ensuremath{\real^n_-}}
\newcommand{\Rn}{\ensuremath{{\mathbb R}^{n-1}}}
\newcommand{\Zn}{\ensuremath{{\mathbb Z}^{n-1}}}
\newcommand{\no}{\ensuremath{\nat_0}}
\newcommand{\ganz}{\ensuremath{\mathbb Z}}
\newcommand{\zn}{\ensuremath{{\mathbb Z}^n}}
\newcommand{\zom}{\ensuremath{{\mathbb Z}_{\Om}}}
\newcommand{\zOm}{\ensuremath{{\mathbb Z}^{\Om}}}
\newcommand{\As}{\ensuremath{A^s_{p,q}}}
\newcommand{\Bs}{\ensuremath{B^s_{p,q}}}
\newcommand{\Fs}{\ensuremath{F^s_{p,q}}}
\newcommand{\Fsr}{\ensuremath{F^{s,\rloc}_{p,q}}}
\newcommand{\nat}{\ensuremath{\mathbb N}}
\newcommand{\Om}{\ensuremath{\Omega}}
\newcommand{\di}{\ensuremath{{\mathrm d}}}
\newcommand{\sn}{\ensuremath{{\mathbb S}^{n-1}}}
\newcommand{\Ac}{\ensuremath{\mathcal A}}
\newcommand{\Acs}{\ensuremath{\Ac^s_{p,q}}}
\newcommand{\Bc}{\ensuremath{\mathcal B}}
\newcommand{\Cc}{\ensuremath{\mathcal C}}
\newcommand{\cc}{{\scriptsize $\Cc$}${}^s (\rn)$}
\newcommand{\ccd}{{\scriptsize $\Cc$}${}^s (\rn, \delta)$}
\newcommand{\Fc}{\ensuremath{\mathcal F}}
\newcommand{\Lc}{\ensuremath{\mathcal L}}
\newcommand{\Mc}{\ensuremath{\mathcal M}}
\newcommand{\Ec}{\ensuremath{\mathcal E}}
\newcommand{\Pc}{\ensuremath{\mathcal P}}
\newcommand{\Efr}{\ensuremath{\mathfrak E}}
\newcommand{\Mfr}{\ensuremath{\mathfrak M}}
\newcommand{\Mbf}{\ensuremath{\mathbf M}}
\newcommand{\Dbb}{\ensuremath{\mathbb D}}
\newcommand{\Lbb}{\ensuremath{\mathbb L}}
\newcommand{\Pbb}{\ensuremath{\mathbb P}}
\newcommand{\Qbb}{\ensuremath{\mathbb Q}}
\newcommand{\Rbb}{\ensuremath{\mathbb R}}
\newcommand{\vp}{\ensuremath{\varphi}}
\newcommand{\hra}{\ensuremath{\hookrightarrow}}
\newcommand{\supp}{\ensuremath{\mathrm{supp \,}}}
\newcommand{\ssupp}{\ensuremath{\mathrm{sing \ supp\,}}}
\newcommand{\dist}{\ensuremath{\mathrm{dist \,}}}
\newcommand{\unif}{\ensuremath{\mathrm{unif}}}
\newcommand{\ve}{\ensuremath{\varepsilon}}
\newcommand{\vk}{\ensuremath{\varkappa}}
\newcommand{\vr}{\ensuremath{\varrho}}
\newcommand{\pa}{\ensuremath{\partial}}
\newcommand{\oa}{\ensuremath{\overline{a}}}
\newcommand{\ob}{\ensuremath{\overline{b}}}
\newcommand{\of}{\ensuremath{\overline{f}}}
\newcommand{\LA}{\ensuremath{L^r\!\As}}
\newcommand{\LcA}{\ensuremath{\Lc^{r}\!A^s_{p,q}}}
\newcommand{\LcdA}{\ensuremath{\Lc^{r}\!A^{s+d}_{p,q}}}
\newcommand{\LcB}{\ensuremath{\Lc^{r}\!B^s_{p,q}}}
\newcommand{\LcF}{\ensuremath{\Lc^{r}\!F^s_{p,q}}}
\newcommand{\Lf}{\ensuremath{L^r\!f^s_{p,q}}}
\newcommand{\La}{\ensuremath{\Lambda}}
\newcommand{\Lob}{\ensuremath{L^r \ob{}^s_{p,q}}}
\newcommand{\Lof}{\ensuremath{L^r \of{}^s_{p,q}}}
\newcommand{\Loa}{\ensuremath{L^r\, \oa{}^s_{p,q}}}
\newcommand{\Lcoa}{\ensuremath{\Lc^{r}\oa{}^s_{p,q}}}
\newcommand{\Lcob}{\ensuremath{\Lc^{r}\ob{}^s_{p,q}}}
\newcommand{\Lcof}{\ensuremath{\Lc^{r}\of{}^s_{p,q}}}
\newcommand{\Lca}{\ensuremath{\Lc^{r}\!a^s_{p,q}}}
\newcommand{\Lcb}{\ensuremath{\Lc^{r}\!b^s_{p,q}}}
\newcommand{\Lcf}{\ensuremath{\Lc^{r}\!f^s_{p,q}}}
\newcommand{\id}{\ensuremath{\mathrm{id}}}
\newcommand{\tr}{\ensuremath{\mathrm{tr\,}}}
\newcommand{\trd}{\ensuremath{\mathrm{tr}_d}}
\newcommand{\trL}{\ensuremath{\mathrm{tr}_L}}
\newcommand{\ext}{\ensuremath{\mathrm{ext}}}
\newcommand{\re}{\ensuremath{\mathrm{re\,}}}
\newcommand{\Rea}{\ensuremath{\mathrm{Re\,}}}
\newcommand{\Ima}{\ensuremath{\mathrm{Im\,}}}
\newcommand{\loc}{\ensuremath{\mathrm{loc}}}
\newcommand{\rloc}{\ensuremath{\mathrm{rloc}}}
\newcommand{\osc}{\ensuremath{\mathrm{osc}}}
\newcommand{\pr}{\pageref}
\newcommand{\wh}{\ensuremath{\widehat}}
\newcommand{\wt}{\ensuremath{\widetilde}}
\newcommand{\ol}{\ensuremath{\overline}}
\newcommand{\os}{\ensuremath{\overset}}
\newcommand{\Li}{\ensuremath{\overset{\circ}{L}}}
\newcommand{\Ai}{\ensuremath{\os{\, \circ}{A}}}
\newcommand{\Ci}{\ensuremath{\os{\circ}{\Cc}}}
\newcommand{\dom}{\ensuremath{\mathrm{dom \,}}}
\newcommand{\SA}{\ensuremath{S^r_{p,q} A}}
\newcommand{\SB}{\ensuremath{S^r_{p,q} B}}
\newcommand{\SF}{\ensuremath{S^r_{p,q} F}}
\newcommand{\Hc}{\ensuremath{\mathcal H}}
\newcommand{\Nc}{\ensuremath{\mathcal N}}
\newcommand{\Lci}{\ensuremath{\overset{\circ}{\Lc}}}
\newcommand{\bmo}{\ensuremath{\mathrm{bmo}}}
\newcommand{\BMO}{\ensuremath{\mathrm{BMO}}}
\newcommand{\cm}{\\[0.1cm]}
\newcommand{\Aa}{\ensuremath{\os{\, \ast}{A}}}
\newcommand{\Ba}{\ensuremath{\os{\, \ast}{B}}}
\newcommand{\Fa}{\ensuremath{\os{\, \ast}{F}}}
\newcommand{\Aas}{\ensuremath{\Aa{}^s_{p,q}}}
\newcommand{\Bas}{\ensuremath{\Ba{}^s_{p,q}}}
\newcommand{\Fas}{\ensuremath{\Fa{}^s_{p,q}}}
\newcommand{\Ca}{\ensuremath{\os{\, \ast}{{\mathcal C}}}}
\newcommand{\Cas}{\ensuremath{\Ca{}^s}}
\newcommand{\Car}{\ensuremath{\Ca{}^r}}
\newcommand{\bl}{$\blacksquare$}

\begin{center}
{\Large Mapping properties of pseudodifferential and Fourier operators}
\\[1cm]
{Hans Triebel}
\\[0.2cm]
Institut f\"{u}r Mathematik\\
Friedrich--Schiller--Universit\"{a}t Jena\\
07737 Jena, Germany
\\[0.1cm]
email: hans.triebel@uni-jena.de
\end{center}

\begin{abstract} 
The composition of the Fourier transform in $\rn$ with a suitable pseudodifferential operator is called a Fourier operator. It is 
compact in appropriate function spaces. The paper deals with its spectral theory. This is based on mapping properties of the Fourier
transform as developed in a preceding paper and related assertions for pseudodifferential operators.
\end{abstract}

\hspace*{20pt}spaces, spectral theory


\section{Introduction}   \label{S1}
The mapping $F_\tau$,
\begin{\eq}   \label{1.1}
(F_\tau f)(x) = \int_{\rn} e^{-i x \xi} \tau (x,\xi) f(\xi) \, \di \xi, \qquad x\in \rn,
\end{\eq}
suitably extended from $S(\rn)$ to $S'(\rn)$ (or to appropriate function spaces) is called a {\em Fourier operator} of the class
$\Phi^\sigma_{1,\delta} (\rn)$ with $\sigma \in \real$ and $0 \le \delta \le 1$ if the {\em symbol} $\tau \in C^\infty (\real^{2n})$
satisfies for all $\alpha \in \nat^n_0$, $\gamma \in \nat^n_0$ and related constants $c_{\alpha, \gamma} \ge 0$,
\begin{\eq}  \label{1.2}
\big| D^\alpha_x D^\gamma_\xi \tau (x,\xi) \big| \le c_{\alpha, \gamma} (1 + |\xi|)^{\sigma -|\gamma| + \delta |\alpha|}, \qquad
x\in \rn, \quad \xi \in \rn.
\end{\eq}
Recall that $T_\tau$,
\begin{\eq}   \label{1.3}
(T_\tau f)(x) = \int_{\rn} e^{-i x \xi} \tau (x,\xi) f(\xi)^\vee \, \di \xi, \qquad x\in \rn,
\end{\eq}
are the well--known corresponding pseudodifferential operators of the H\"{o}rmander class $\Psi^\sigma_{1, \delta} (\rn)$. Here
$f^\vee (\xi) = (F^{-1} f)(\xi)$ is the inverse Fourier transform, whereas the Fourier transform $\wh{f}(\xi) = (Ff)(\xi)$ is given by
\eqref{1.1} with $\tau =1$. The elaborated theory of pseudodifferential operators of the class $\Psi^\sigma_{1,\delta} (\rn)$ and the
detailed study of the Fourier transform $F$ in suitable function spaces $B^s_p (\rn) = B^s_{p,p} (\rn)$ in \cite{T21} suggest to deal
with $F_\tau \in \Phi^\sigma_{1, \delta} (\rn)$ as the composition
\begin{\eq}  \label{1.4}
F_\tau = T_\tau \circ F, \qquad T_\tau \in \Psi^\sigma_{1,\delta} (\rn).
\end{\eq}
This is the main topic of the paper. It comes out that $F_\tau$ is compact in some function spaces $B^s_p (\rn)$. Of interest is the
distribution of its non--zero eigenvalues. The respective Theorem \ref{T4.5} in Section \ref{S4.2} is our main result. As far as 
compact mappings  generated by the Fourier transform $F$ in some function spaces of type $B^s_p (\rn)$  are concerned  we rely on
\cite{T21}. Mappings of the pseudodifferential operator $T_\tau$  in the function spaces
\begin{\eq}   \label{1.5}
\As (\rn), \qquad A \in \{B,F \}, \quad s\in \real  \quad \text{and} \quad 0<p,q \le \infty
\end{\eq}
($p<\infty$ for the $F$--spaces) have been treated in the 1980s and 1990s based on the technicalities available at that time. As a preparation of our main result we return in Section \ref{S3} to this topic and offer a new proof based on wavelets. The related
Theorem \ref{T3.1} complemented by Proposition \ref{P3.3} might be of some self--contained interest. They recover and extend what is
already known adapted to our later needs. Finally we discuss in Section \ref{S4.3} so--called {\em Fourier heat operators} as an
example of the general assertions.

\section{Function spaces}    \label{S2}
\subsection{Definitions}   \label{S2.1}
We use standard notation. Let $\nat$ be the collection of all natural numbers and $\no = \nat \cup \{0 \}$. Let $\rn$ be Euclidean $n$-space where
$n\in \nat$. Put $\real = \real^1$, whereas $\comp$ stands for the complex plane.
Let $S(\rn)$ be the Schwartz space of all complex-valued rapidly decreasing infinitely differentiable functions on $\rn$ and let $S' (\rn)$ be the dual space of all tempered distributions on \rn.
Furthermore, $L_p (\rn)$ with $0< p \le \infty$, is the standard complex quasi-Banach space with respect to the Lebesgue measure, quasi-normed by
\begin{\eq}   \label{2.1}
\| f \, | L_p (\rn) \| = \Big( \int_{\rn} |f(x)|^p \, \di x \Big)^{1/p}
\end{\eq}
with the standard modification if $p=\infty$.  
As usual, $\ganz$ is the collection of all integers; and $\zn$ where $n\in \nat$ denotes the
lattice of all points $m= (m_1, \ldots, m_n) \in \rn$ with $m_k \in \ganz$. 

If $\vp \in S(\rn)$ then
\begin{\eq}  \label{2.2}
\wh{\vp} (\xi) = (F \vp)(\xi) = (2\pi )^{-n/2} \int_{\rn} e^{-ix \xi} \vp (x) \, \di x, \qquad \xi \in  \rn,
\end{\eq}
denotes the Fourier transform of \vp. As usual, $F^{-1} \vp$ and $\vp^\vee$ stand for the inverse Fourier transform, given by the right-hand side of
\eqref{2.2} with $i$ in place of $-i$. Here $x \xi = \sum^n_{j=1} x_j \xi_j$, $x\in \rn$, $\xi \in \rn$,
 stands for the scalar product in \rn. Both $F$ and $F^{-1}$ are extended to $S'(\rn)$ in the
standard way. Let $\vp_0 \in S(\rn)$ with
\begin{\eq}   \label{2.3}
\vp_0 (x) =1 \ \text{if $|x|\le 1$} \quad \text{and} \quad \vp_0 (x) =0 \ \text{if $|x| \ge 3/2$},
\end{\eq}
and let
\begin{\eq}   \label{2.4}
\vp_k (x) = \vp_0 (2^{-k} x) - \vp_0 (2^{-k+1} x ), \qquad x\in \rn, \quad k\in \nat.
\end{\eq}
Since
\begin{\eq}   \label{2.5}
\sum^\infty_{j=0} \vp_j (x) =1 \qquad \text{for} \quad x\in \rn,
\end{\eq}
$\vp =\{ \vp_j \}^\infty_{j=0}$ forms a dyadic resolution of unity. The entire analytic functions $(\vp_j \wh{f} )^\vee (x)$ make sense pointwise in $\rn$ for any $f\in S'(\rn)$. 

\begin{definition}   \label{D2.1}
Let $\vp = \{ \vp_j \}^\infty_{j=0}$ be the above dyadic resolution  of unity.
\\[0.1cm]
{\upshape (i)} Let
\begin{\eq}   \label{2.6}
0<p \le \infty, \qquad 0<q \le \infty, \qquad s \in \real.
\end{\eq}
Then $\Bs (\rn)$ is the collection of all $f \in S' (\rn)$ such that
\begin{\eq}   \label{2.7}
\| f \, | \Bs (\rn) \|_{\vp} = \Big( \sum^\infty_{j=0} 2^{jsq} \big\| (\vp_j \widehat{f})^\vee \, | L_p (\rn) \big\|^q \Big)^{1/q} 
\end{\eq}
is finite $($with the usual modification if $q= \infty)$. 
\\[0.1cm]
{\upshape (ii)} Let
\begin{\eq}   \label{2.8}
0<p<\infty, \qquad 0<q \le \infty, \qquad s \in \real.
\end{\eq}
Then $F^s_{p,q} (\rn)$ is the collection of all $f \in S' (\rn)$ such that
\begin{\eq}   \label{2.9}
\| f \, | F^s_{p,q} (\rn) \|_{\vp} = \Big\| \Big( \sum^\infty_{j=0} 2^{jsq} \big| (\vp_j \wh{f})^\vee (\cdot) \big|^q \Big)^{1/q} \big| L_p (\rn) \Big\|
\end{\eq}
is finite $($with the usual modification if $q=\infty)$.
\end{definition}

\begin{remark}    \label{R2.2}
The theory of these spaces $\As (\rn)$, $A \in \{B,F \}$, $s\in \real$, $0<p \le \infty$ ($p<\infty$ for $F$--spaces) has been 
developed in many papers and books, including \cite{T83}, \cite{T92}, \cite{T06} and \cite{T20}. There one finds detailed (historical)
references. It is well known that these spaces are independent of the chosen resolution of unity $\vp$ (equivalent quasi--norms). This
justifies our omission of the subscript $\vp$ in \eqref{2.7} and \eqref{2.9} in the sequel. Mapping properties of pseudodifferential
operators will be discussed in Section \ref{S3} in the context of the spaces $\As(\rn)$ and by passing of $F^s_{\infty,q} (\rn)$ in
full generality. This might be of some self--contained interest. But as already indicated in the Introduction it is mainly a 
preparation of the theory of the so--called Fourier operators as developed in Section \ref{S4}. Then we will specify the above spaces 
$\As (\rn)$ to the distinguished Besov spaces
\begin{\eq}  \label{2.10}
B^s_p (\rn) = B^s_{p,p} (\rn) = F^s_{p,p} (\rn), \qquad s\in \real, \quad 0<p \le \infty,
\end{\eq}
with the H\"{o}lder--Zygmund spaces
\begin{\eq}    \label{2.11}
B^s_\infty (\rn) = \Cc^s (\rn), \qquad s \in \real,
\end{\eq}
as special cases.
\end{remark}

\subsection{Wavelet characterizations}    \label{S2.2}
We assume that the reader is familiar with the basic assertions for the spaces $\As (\rn)$, including wavelet characterizations. But
we need in Section \ref{S3} a few rather specific properties. This may justify to repeat some basic definitions and assertions. We
follow \cite[Section 1.2.1, pp.\,7--10]{T20}. There one finds explanations, discussions and, in particular, references. This will not
be repeated here. As usual, $C^{u} (\real)$ with $u\in
\nat$ collects all bounded complex-valued continuous functions on $\real$ having continuous bounded derivatives up to order $u$ inclusively. Let
\begin{\eq}   \label{2.12}
\psi_F \in C^{u} (\real), \qquad \psi_M \in C^{u} (\real), \qquad u \in \nat,
\end{\eq}
be {\em real} compactly supported Daubechies wavelets with
\begin{\eq}   \label{2.13}
\int_{\real} \psi_M (x) \, x^v \, \di x =0 \qquad \text{for all $v\in \no$ with $v<u$.}
\end{\eq}
One extends these wavelets from $\real$ to $\rn$ by the usual multiresolution procedure. Let $n\in \nat$ and let
\begin{\eq}   \label{2.14}
G = (G_1, \ldots, G_n) \in G^0 = \{F,M \}^n
\end{\eq}
which means that $G_r$ is either $F$ or $M$. Furthermore, let
\begin{\eq}   \label{2.15}
G= (G_1, \ldots, G_n) \in G^* = G^j = \{F, M \}^{n*}, \qquad j \in \nat,
\end{\eq}
which means that $G_r$ is either $F$ or $M$, where $*$ indicates that at least one of the components of $G$ must be an $M$. Hence $G^0$ has $2^n$ elements, whereas $G^j$ with $j\in \nat$ and $G^*$ have $2^n -1$ elements. Let
\begin{\eq}   \label{2.16}
\psi^j_{G,m} (x) = \prod^n_{l=1} \psi_{G_l} \big(2^j x_l -m_l \big), \qquad G\in G^j, \quad m \in \zn, \quad x\in \rn,
\end{\eq}
where (now) $j \in \no$. We always assume that $\psi_F$ and $\psi_M$ in \eqref{2.12} have $L_2$--norm 1. Then 
\begin{\eq}   \label{2.17}
 \big\{ 2^{jn/2} \psi^j_{G,m}: \ j \in \no, \ G\in G^j, \ m \in \zn \big\}
\end{\eq}
is an {\em orthonormal basis} in $L_2 (\rn)$ (for any $u\in \nat$) and
\begin{\eq}   \label{2.18}
f = \sum^\infty_{j=0} \sum_{G \in G^j} \sum_{m \in \zn} \lambda^{j,G}_m \, \psi^j_{G,m}
\end{\eq}
with
\begin{\eq}   \label{2.19}
\lambda^{j,G}_m = \lambda^{j,G}_m (f) = 2^{jn} \int_{\rn} f(x) \, \psi^j_{G,m} (x) \, \di x = 2^{jn} \big(f, \psi^j_{G,m} \big)
\end{\eq}
is the corresponding expansion. Let $\chi_{j,m}$ be the characteristic function of the cube 
\begin{\eq}   \label{2.20}
Q_{j,m} = 2^{-j} m + 2^{-j} (0,1)^n, \qquad j\in \no, \quad m\in \zn.
\end{\eq}

\begin{definition}   \label{D2.3}
Let
\begin{\eq}   \label{2.21}
\lambda = \big\{ \lambda^{j,G}_m \in \comp: \ j \in \no, \ G \in G^j, \ m \in \zn \big\}.
\end{\eq}
Let $0<p,q \le \infty$ and $s\in \real$.
Then
\begin{\eq}   \label{2.22}
b^s_{p,q} (\rn) = \big\{ \lambda: \ \| \lambda \, | b^s_{p,q} (\rn) \| < \infty \big\}
\end{\eq}
with
\begin{\eq}   \label{2.23}
\| \lambda \, | b^s_{p,q} (\rn) \| = \Big( \sum^\infty_{j=0} 2^{j(s- \frac{n}{p})q} \sum_{G \in G^j} \Big( \sum_{m \in \zn} |\lambda^{j,G}_m|^p \Big)^{q/p} \Big)^{1/q}
\end{\eq}
and
\begin{\eq}   \label{2.24}
f^s_{p,q} (\rn) = \big\{ \lambda: \ \| \lambda \, | f^s_{p,q} (\rn) \| < \infty \big\}
\end{\eq}
with
\begin{\eq}   \label{2.25}
\| \lambda \, | f^s_{p,q} (\rn) \| = \Big\| \Big( \sum_{\substack{j\in \no, G\in G^j,\\ m\in \zn}}
2^{jsq} \big| \lambda^{j,G}_m \, \chi_{j,m} (\cdot) \big|^q \Big)^{1/q} \big| L_p (\rn)
\Big\|
\end{\eq}
$($usual modifications if $\max(p,q) =\infty)$.
\end{definition}

We still follow \cite{T20} and the references given there. Let $n\in \nat$ and, as usual,
\begin{\eq}   \label{2.26}
\sigma_p^n = n \Big( \max \big( \frac{1}{p}, 1 \big) - 1 \Big), \qquad \sigma^n_{p,q} = n \Big( \max \big( \frac{1}{p}, \frac{1}{q}, 1 \big) -1 \Big)
\end{\eq}
where $0<p,q \le \infty$.   

\begin{proposition}   \label{P2.4}
{\upshape (i)} Let $0<p \le \infty$, $0<q \le \infty$, $s\in \real$ and
\begin{\eq}   \label{2.27}
u > \max (s, \sigma^n_p -s).
\end{\eq}
Let $f \in S' (\rn)$. Then $f \in \Bs (\rn)$ if, and only if, it can be represented as
\begin{\eq}   \label{2.28}
f= \sum_{\substack{j\in \no, G\in G^j, \\ m\in \zn}}
 \lambda^{j,G}_m \, \psi^j_{G,m}, \qquad \lambda \in b^s_{p,q} (\rn),
\end{\eq}
the unconditional convergence being in $S' (\rn)$. The representation \eqref{2.28} is unique,
\begin{\eq}  \label{2.29}
\lambda^{j,G}_m = \lambda^{j,G}_m (f) =  2^{jn} \big( f, \psi^j_{G,m} \big)
\end{\eq}
and
\begin{\eq}   \label{2.30}
I: \quad f \mapsto \big\{ \lambda^{j,G}_m (f) \big\}
\end{\eq}
is an isomorphic map of $\Bs (\rn)$ onto $b^s_{p,q} (\rn)$.
\\[0.1cm]
{\upshape (ii)} Let $0<p < \infty$, $0<q\le \infty$, $s\in \real$ and
\begin{\eq}   \label{2.31}
u > \max (s, \sigma^n_{p,q} -s).
\end{\eq}
Let $f \in S' (\rn)$. Then $f \in \Fs (\rn)$ if, and only if, it can be represented as
\begin{\eq}   \label{2.32}
f = \sum_{\substack{j\in \no, G\in G^j, \\ m\in \zn}} \lambda^{j,G}_m \, \psi^j_{G,m}, \qquad \lambda \in f^s_{p,q} (\rn),
\end{\eq}
the unconditional convergence being in $S' (\rn)$. The representation \eqref{2.32}
 is unique with \eqref{2.29}. Furthermore $I$ in \eqref{2.30} is an isomorphic map of
$\Fs (\rn)$ onto $f^s_{p,q} (\rn)$.
\end{proposition}

\begin{remark}   \label{R2.5}
This coincides with the relevant parts of \cite[Proposition 1.11, pp.\,9--10]{T20}. There one finds also related references. We use
the above substantial  assertion as a tool. This may explain why we did not promote this proposition to a theorem.
\end{remark}

\section{Pseudodifferential operators}    \label{S3}
\subsection{Preliminaries}    \label{S3.1}
Let $T_\tau$,
\begin{\eq}   \label{3.1}
(T_\tau f)(x) = \int_{\rn} e^{ix \xi} \tau (x,\xi) \wh{f} (\xi) \, \di \xi, \qquad x\in \rn,
\end{\eq}
be a pseudodifferential operator of the H\"{o}rmander class $\Psi^\sigma_{1,\delta} (\rn)$ with $\sigma \in \real$ and $0 \le \delta
\le 1$ where the symbol $\tau (x,\xi) \in C^\infty (\real^{2n} )$ satisfies for some constants $c_{\alpha, \gamma} \ge 0$,
\begin{\eq}   \label{3.2}
\big| D^\alpha_x D^\gamma_\xi \tau (x, \xi) \big| \le c_{\alpha, \gamma} (1+ |\xi|)^{\sigma - |\gamma| + \delta |\alpha|}, \qquad
x\in \rn, \quad \xi \in \rn,
\end{\eq}
$\alpha \in \nat^n_0$, $\gamma \in \nat^n_0$. Here $\nat^n_0 = \{ x\in \rn: \, x=(x_1, \ldots, x_n ), \, x_k \in \no \}$. To be in
agreement with the literature we changed the roles of $F$ and $F^{-1}$ in \eqref{1.3}. But this is immaterial. These classes of
operators attracted a lot of attention since the 1970s (and even earlier). The standard references are \cite{Tay81}, \cite{Hor85}
amd the related parts of \cite{Ste93}. There one finds basic assertions and many applications. They combine multiplications, partial
differential equations, their inverses (if exist) and related integral operators. Mapping properties of these operators in the spaces
$\As (\rn)$ according to \eqref{1.5} ($p<\infty$ for $F$--spaces) have been studied in the 1980s and 1990s based on the technicalities
available at that time.  Related references will be given below in Remark \ref{R3.2}. We return to this topic in the context of wavelet
expansions as described in Section \ref{S2.2}. We hope that the interplay of the miraculous properties of wavelets with the conditions
\eqref{3.2} for the symbols sheds new light on this theory. On the other hand, as already indicated in the Introduction, it is the
main concern of this paper to develop a spectral theory of the Fourier operators $F_\tau$ according to \eqref{1.1} based on the 
composition \eqref{1.4}.

We collect a few classical assertions about pseudodifferential operators. If
\begin{\eq}   \label{3.3}
T_{\tau_1} \in \Psi^{\mu_1}_{1, \delta} (\rn) \quad \text{and} \quad T_{\tau_2} \in \Psi^{\mu_2}_{1, \delta} (\rn) \quad \text{with $\mu_1, \mu_2 \in \real$ and $0\le \delta <1$}
\end{\eq}
then one has for the composition operator
\begin{\eq}    \label{3.4}
T_\tau = T_{\tau_1} \circ T_{\tau_2} \in \Psi^{\mu_1 + \mu_2}_{1, \delta} (\rn).
\end{\eq}
This is one of the crucial observations of the theory of pseudodifferential operators covered by \cite[pp.\,71, 94]{Hor85}. It is no
longer true for the so--called {\em exotic} class $\Psi^\mu_{1,1} (\rn)$. Recall that
\begin{\eq}   \label{3.5}
I_{\vr} \As (\rn) = A^{s+\vr}_{p,q} (\rn), \qquad \vr \in \real,
\end{\eq} 
for $A \in \{B,F \}$, $s\in \real$ and $0<p,q \le \infty$, where $I_{\vr}$ is the well--known lift
\begin{\eq}   \label{3.6}
I_{\vr}f = \big( \langle \xi \rangle^{-\vr} \wh{f}\, \big)^\vee \quad \text{with} \quad \langle \xi \rangle = \big(1+ |\xi|^2 \big)^{1/2}, \quad \xi \in \rn, \quad \vr \in \real,
\end{\eq}
\cite[Theorem 1.22, p.\,16]{T20}. The observations \eqref{3.3}, \eqref{3.4} combined with $I_{\vr} \in \Psi^{-\vr}_{1, \delta} (\rn)$,
$0 \le \delta <1$ show that
\begin{\eq}   \label{3.7}
T_{\tau} \circ I_{\sigma} \in \Psi^0_{1,\delta} (\rn) \qquad \text{for} \quad T_{\tau} \in \Psi^\sigma_{1,\delta} (\rn), \quad \sigma \in \real, \quad 0 \le \delta <1.
\end{\eq}
Combined with \eqref{3.5} it follows that it is sufficient
for fixed $p,q$ and $0 \le \delta <1$ to concentrate on mapping properties of the operators
\begin{\eq}   \label{3.8}
T_\tau \in \Psi^0_{1,\delta} (\rn) \quad \text{in} \quad \As(\rn) \quad \text{for some fixed $s\in \real$}.
\end{\eq}
The situation is different for the exotic class $\Psi^0_{1,1} (\rn)$. But at least some rescue comes from the dual operator $T'$,
\begin{\eq}   \label{3.9}
(T  \vp, \psi) = (\vp, T' \psi) \qquad \text{with, say,} \quad \vp \in S(\rn), \quad \psi \in S(\rn).
\end{\eq}
It is a further basic assertion of the theory of pseudodifferential operators that
\begin{\eq}  \label{3.10}
T' \in \Psi^0_{1,\delta} (\rn) \qquad \text{if} \quad T \in \Psi^0_{1,\delta} (\rn) \quad \text{and} \quad 0 \le \delta <1.
\end{\eq}
We refer the reader to \cite[Theorem 18.1.7, pp.\,70, 94]{Hor85}. One may also consult \cite[Theorems 4.1, 4.2, p.\,45]{Tay81}. Again
there is no counterpart for the exotic class $\Psi^0_{1,1} (\rn)$. Nevertheless duals of exotic pseudodifferential operators will play
some role below. Then we give also more specific references. In connection with \eqref{3.9} we remark that one has
\begin{\eq}   \label{3.11}
T_\tau: \quad S(\rn) \hra S(\rn)
\end{\eq}
for all operators $T_\tau \in \Psi^0_{1,\delta} (\rn)$ with $0 \le \delta \le 1$. This again is a
classical assertion covered by \cite[pp.\,68, 94]{Hor85}. But it will be of some use for us to insert a short proof. Let $f\in S(\rn)$
and $\alpha \in \nat^n_0$. Then it follows from \eqref{3.1} that
\begin{\eq}   \label{3.12}
\begin{aligned}
x^\alpha (T_\tau f)(x) &= (-i)^{|\alpha|} \int_{\rn} \big( D^\alpha_\xi e^{ix \xi} \big) \, \tau (x, \xi) \, \wh{f} (\xi) \, \di \xi \\
&= i^{|\alpha|} \int_{\rn} e^{ix \xi} \, D^\alpha_\xi \big[ \tau (x,\xi) \wh{f} (\xi) \big] \, \di \xi
\end{aligned}
\end{\eq}
where $x^\alpha = \prod^n_{j=1} x^{\alpha_j}_j$, $x= (x_1, \ldots, x_n)$, $\alpha \in \nat^n_0$,  and
\begin{\eq}   \label{3.13}
D^\alpha_x (T_\tau f)(x) = \sum_{\beta + \gamma = \alpha} d_{\beta, \gamma} \int_{\rn} \big( D^\beta_x e^{ix \xi} \big) \cdot 
\big(D^\gamma_x \, \tau (x, \xi) \big) \wh{f} (\xi) \, \di \xi.
\end{\eq}
Then \eqref{3.11} follows from \eqref{3.2} and the usual (semi--)norms generating the topology in $S(\rn)$. 

\subsection{Mapping properties}   \label{S3.2}
After the above preparations we prove mapping properties of the pseudodifferential operators $T_\tau$ according to \eqref{3.1}, 
\eqref{3.2} belonging to the H\"{o}rmander class $\Psi^0_{1, \delta} (\rn)$, $0 \le \delta \le 1$.
\begin{theorem}   \label{T3.1}
Let
\begin{\eq}   \label{3.14}
\As (\rn) \qquad \text{with} \quad A\in \{B,F \}, \quad s\in \real \quad \text{and} \quad 0<p,q \le \infty
\end{\eq}
$(p<\infty$ for the $F$--spaces$)$ as introduced in Definition \ref{D2.1}.
Let either $T_\tau \in \Psi^0_{1,\delta} (\rn)$ with $0 \le \delta <1$ or $T_\tau \in \Psi^0_{1,1} (\rn)$ such that
$T'_\tau \in \Psi^0_{1,1} (\rn)$. Then $T_\tau$ generates a continuous mapping from $\As (\rn)$ into itself,
\begin{\eq}   \label{3.15}
T_\tau: \quad \As (\rn) \hra \As (\rn).
\end{\eq}
\end{theorem}

\begin{proof}
{\em Step 1.} Let first, in addition,
\begin{\eq}   \label{3.16}
s> \sigma^n_p \ \text{for $B$--spaces} \qquad \text{and} \qquad s> \sigma^n_{p,q} \ \text{for $F$--spaces}
\end{\eq}
with $\sigma^n_p$ and $\sigma^n_{p,q}$ as in \eqref{2.26}. Then the elements belonging to $\As (\rn)$ can be characterized by atomic
representations where the corresponding $L_\infty$--normalized sufficiently smooth atoms $a_{j,m} (x)$, $j\in \no$, $m\in \zn$, are
supported by $d Q_{j,m}$ with $d>1$ and $Q_{j,m}$ as in \eqref{2.20}. The restriction \eqref{3.16} ensures that no moment conditions
for the atoms are requested. The related sequence spaces are similar as in Definition \ref{D2.3} without the summation over $G\in G^j$.
Details and references (including precise formulations) may be found in \cite[Section 1.1.2, pp.\,4--5]{T08}. Our method is in 
principle rather straightforward. We rely on the wavelet representation
\begin{\eq}   \label{3.17}
f = \sum^\infty_{j=0} \sum_{G \in G^j} \sum_{m\in \zn} 2^{jn} \, \big( f, \psi^j_{G,m} \big) \, \psi^j_{G,m}
\end{\eq}
according to Proposition \ref{P2.4}
and ask whether $T_\tau \psi^j_{G,m}$ can be reduced to admitted 
 $L_\infty$--normalized atoms (without moment conditions). Then the coefficients
of the resulting atomic expansions originate from their counterparts $2^{jn} (f, \psi^j_{G,m} )$ and belong, as a consequence, to the
desired sequence spaces. This justifies 
\begin{\eq}   \label{3.18}
T_\tau f = \sum^\infty_{j=0} \sum_{G \in G^j} \sum_{m\in \zn} 2^{jn} \, \big( f, \psi^j_{G,m} \big) \,T_\tau \psi^j_{G,m}
\end{\eq}
and \eqref{3.15} under the restriction \eqref{3.16}. Let $j\in \no$. Then it follows from \eqref{2.2} and \eqref{2.16} that
\begin{\eq}   \label{3.19}
\big( \psi^j_{G,m} \big)^{\wedge} (\xi) =  2^{-jn} \, e^{-i 2^{-j}m \xi} \big( \psi^0_{G,0} \big)^{\wedge} (2^{-j} \xi), \qquad \xi \in
\rn.
\end{\eq}
Inserted in \eqref{3.1} one obtains
\begin{\eq}   \label{3.20}
\big(T_\tau \psi^j_{G,m} \big)(x) = \int_{\rn} e^{i 2^j x \xi - im \xi} \, \tau (x, 2^j \xi) \, \big( \psi^0_{G,0} \big)^{\wedge} (\xi) \, \di \xi.
\end{\eq}
Let $\psi$ be a compactly supported $C^\infty$ in $\rn$ such that
\begin{\eq}   \label{3.21}
1 = \sum_{k\in \zn} \psi (x-k), \qquad x \in \rn.
\end{\eq}
Let
\begin{\eq}   \label{3.22}
b^{G,k}_{j,m} (x) = \psi(x-k) \int_{\rn} e^{ix \xi} \, \tau \big( 2^{-j} (x+m), 2^j \xi \big) \, \big( \psi^0_{G,0} \big)^\wedge (\xi) \,\di \xi,
\end{\eq}
$x\in \rn$, where $j\in \no$, $G \in G^j$, $k\in \zn$, $m\in \zn$. If $j \in \nat$ in \eqref{3.19} then at
least one $G_l$ in $G = (G_1, \ldots, G_n )\in G^j$ is an $M$ and it follows  from \eqref{2.12}, \eqref{2.13} and \eqref{2.16} that
\begin{\eq}   \label{3.23}
\big| D^\alpha_\xi \big( \psi^0_{G,0} \big)^\wedge (\xi) \big| \le c\, \frac{|\xi|^v}{(1+ |\xi|)^w}, \qquad \xi \in \rn, \quad 0 \le
|\alpha| \le L,
\end{\eq}
where $L\in \no$, $v\in \nat$ and $w\in \nat$ are (independently) at our disposal. If $j =0$ in \eqref{3.19} then one has \eqref{3.23}
with $L\in \no$, $v=0$ and $w\in \nat$. In any case it follows from  \eqref{3.2} that the integral in \eqref{3.22} makes sense if $v$
and $w$ are chosen  appropriately. This applies also to what follows. Furthermore one has by \eqref{3.2} with $\sigma =0$ and adapted
modifications of \eqref{3.12}, \eqref{3.13} that the integrals in \eqref{3.22} are sufficiently smooth and that the needed derivatives
decay uniformly in $j\in \no$ and $m\in \zn$ strongly enough to justify the subsequent arguments. This is the point where one needs
\eqref{3.23} for given $v\in \nat$ if $j\in \nat$ and for given $w$ if $j\in \no$, in addition to $L\in \no$. This applies also to 
$x^\alpha b^{G.k}_{j,m} (x)$ in \eqref{3.22} and by the support property of $\psi (x-k)$ to
\begin{\eq}   \label{3.24}
\langle k \rangle^D b^{G,,k}_{j,m} (x), \qquad \langle k \rangle = (1 + |k|^2)^{1/2}, \quad k\in \zn,
\end{\eq}
where $D>0$ is at our disposal. By \eqref{3.22} one can rewrite \eqref{3.20} as
\begin{\eq}   \label{3.25}
\big( T_\tau \psi^j_{G,m} \big)(x) = \sum_{k\in \zn} a^{G,k}_{j,m} (x) \qquad \text{with} \quad a^{G,k}_{j,m} (x) = b^{G,k}_{j,m} (2^j x -m),
\end{\eq}
where $j\in \no$, $m\in \zn$ and $G \in G^j$. Using in addition \eqref{3.24} it follows that
\begin{\eq}   \label{3.26}
\langle k \rangle^D a^{G,k}_{j,m} (x) \qquad \text{are $L_\infty$--normalized}
\end{\eq}
classical  atoms for any prescribed $D>0$ located at $2^{-j} (m+k)$ (no moment conditions are needed for the spaces $\As (\rn)$ with
$s$ as in \eqref{3.16}). Recall that $\As (\rn)$ and the related sequence spaces are $u$--Banach spaces with $u= \min(1,p,q)$. Choosing
now $D$ in \eqref{3.26} sufficiently large, then \eqref{3.18} is an atomic decomposition in $\As (\rn)$ based on the same coefficients
$2^{jn} (f, \psi^j_{G,m} )$ as in \eqref{3.17}. This proves \eqref{3.15} under the restriction \eqref{3.16}.
\cm
{\em Step 2.} We extend the above assertion to the remaining cases. Let as usual
\begin{\eq}    \label{3.27}
\frac{1}{p} + \frac{1}{p'} = \frac{1}{q} + \frac{1}{q'} = 1 \qquad \text{where} \quad 1 \le p,q \le \infty
\end{\eq}
and $s\in \real$. Then one has in the framework of the dual pairing $\big( S(\rn), S'(\rn) \big)$ the duality
\begin{\eq}   \label{3.28}
\Bs (\rn)' = B^{-s}_{p',q'} (\rn) \qquad \text{if} \quad 1 \le p,q <\infty
\end{\eq}
and
\begin{\eq}   \label{3.29}
\Fs (\rn)' = F^{-s}_{p', q'} (\rn) \qquad \text{if} \quad 1<p<\infty, \quad 1\le q <\infty.
\end{\eq}
One may consult \cite[Theorems 2.11.2, p.\,178]{T83}, complemented by \cite[Proposition, p.\,20]{RuS96} (as far as $q=1$
in \eqref{3.29} is concerned). There one finds further assertions about dual spaces and related references. Let $0\le \delta <1$ and
$s>0$. Then it follows from \eqref{3.28}, \eqref{3.10} and Step 1 that
\begin{\eq}   \label{3.30}
\begin{aligned}
\| T_\tau f \, | B^{-s}_{p',q'} (\rn) \| & = \sup \big\{ |(T_\tau f,g)|: \ \| g \, | \Bs (\rn) \| \le 1 \big\} \\
&= \sup \big\{ |(f, T'_\tau g)|: \ \|g \, | \Bs (\rn) \| \le 1 \big\} \\
&\le c \, \|f \, | B^{-s}_{p', q'} (\rn) \|.
\end{aligned}
\end{\eq}
Similarly for the $F$--spaces based on \eqref{3.29}. This can be extended to the exotic class $\Psi^0_{1,1} (\rn)$ where we assumed that both $T_\tau$ and $T'_\tau$ belong to $\Psi^0_{1,1} (\rn)$. The remaining cases can now be incorporated by interpolation. 
According to \cite[Theorem 8.5, pp.\,98, 134]{FrJ90} one has
\begin{\eq}   \label{3.31}
\Fs (\rn) = \big\langle F^{s_0}_{p_0, q_0} (\rn), F^{s_1}_{p_1, q_1} (\rn), \theta \big\rangle
\end{\eq}
for the so--called $\pm$method of interpolation theory where $0<\theta <1$, $s_0 \in \real$, $s_1 \in \real$, $0<p_0, q_0, p_1, q_1
\le \infty$ and
\begin{\eq}   \label{3.32}
s= (1-\theta)s_0 + \theta s_1, \quad \frac{1}{p} = \frac{1-\theta}{p_0} + \frac{\theta}{p_1}, \quad  
\frac{1}{q} = \frac{1-\theta}{q_0} + \frac{\theta}{q_1},
\end{\eq}
(including $F^s_{\infty,q} (\rn)$ what will be of some use for us below). First we fix $1<p_0 = p_1 =p <\infty$ and apply \eqref{3.31},
\eqref{3.32} to what is already known. This can be characterized by line segments in an $\big( \frac{1}{q}, s \big)$--diagram 
connecting related regions covered by \eqref{3.16} and \eqref{3.29}. This show that the desired mapping property for $T_\tau$ is also
valid for
\begin{\eq}   \label{3.33}
T_\tau: \quad \Fs (\rn) \hra \Fs (\rn), \qquad s\in \real, \quad 1<p<\infty, \quad0<q \le \infty.
\end{\eq}
Secondly  we fix $0<q_0 = q_1 =q \le \infty$ and interpolate the $F$--spaces covered by \eqref{3.16} and \eqref{3.33} what again can
be illuminated in an $\big( \frac{1}{p}, s \big)$--diagram. This proves \eqref{3.15} for $A=F$. The  related assertion for the spaces
$\Bs (\rn)$ with $p<\infty$ follows from the real interpolation
\begin{\eq}   \label{3.34}
\Bs (\rn) = \big( F^{s_0}_{p,q} (\rn), F^{s_1}_{p,q} (\rn) \big)_{\theta,q}, \qquad s_0 \not= s_1.
\end{\eq}
The duality \eqref{3.28} with $p=1$ and a further real interpolation extends the desired assertion to the remaining spaces 
$B^s_{\infty,q} (\rn)$, $0<q \le \infty$.
\end{proof}

\begin{remark}   \label{R3.2}
A different proof of Theorem \ref{T3.1} for $T_\tau \in \Psi^0_{1, \delta} (\rn)$ with $0 \le \delta <1$ may be found in \cite[Theorem 6.2.2,
pp.\,258--261]{T92} based on \cite{Pai83} and \cite{T87}. It relies on local means and related mappings in $\As (\rn)$ for large $s$.
Afterwards one can extend this assertion as in Step 2 of the above proof. But alternatively one can rely on the lifts $I_{\vr} \in
\Psi^{-\vr}_{1, \delta} (\rn)$, $0\le \delta <1$ in \eqref{3.5}, \eqref{3.6} and \eqref{3.4} specified to
\begin{\eq}   \label{3.35}
I_{-\vr} \circ T_\tau \circ I_{\vr} \in \Psi^0_{1,\delta} (\rn) \quad \text{if, and only if,} \quad T_\tau \in \Psi^0_{1,\delta} (\rn),
\quad \delta <1.
\end{\eq}
This does not work for the exotic class $\Psi^0_{1,1} (\rn)$. However Step 1 of the proof shows that one has
\begin{\eq}   \label{3.36}
T_\tau: \quad \Fs (\rn) \hra \Fs (\rn) \quad \text{if} \quad 0<p<\infty, \quad 0<q\le \infty, \quad s > \sigma^n_{p,q},
\end{\eq}
for any $T_\tau \in \Psi^0_{1,1} (\rn)$ without any assumption about its dual. This had already been observed in \cite{Run85}.
On the other hand it is well known that there are operators
$T_\tau \in \Psi^0_{1,1} (\rn)$ which are not continuous in $L_2 (\rn)$ (and $L_p (\rn)$ with $1<p<\infty$). An example may be found in
\cite[pp.\,272--274]{Ste93}. The significant role played by the assumption that both $T_\tau$ and its dual $T'_\tau$ belong to $\Psi^0_{1,1}
(\rn)$ had been discussed in \cite{Bou88} . A modification goes back to \cite{Hor88}, \cite{Hor89} dealing with mapping properties of
related operators $T_\tau \in \Psi^0_{1,1} (\rn)$ in the spaces $H^s (\rn) = H^s_2 (\rn) = F^s_{2,2} (\rn)$ with $s<0$. This has been extended
in \cite{Tor90} to the spaces $\As (\rn)$ in \eqref{3.14}. A related homogeneous version based on Frazier--Jawerth frames according to
\cite{FrJ90} may be found in \cite{GrT99}.
One may also consult \cite[\S 5]{Tor91}. In other words, the above theorem
is more or less known since 3 decades. But at least some arguments of our approach might be new and efficient. They shed some light
on the interplay of wavelets, especially \eqref{3.23}, and the condition \eqref{3.2} for the underlying symbols.
\end{remark}

Step 1 of the proof of the above theorem relies on the observation \eqref{3.25} where $D>0$ in \eqref{3.26} is at our disposal. But
this can be used in further function spaces admitting  the wavelet representation \eqref{3.17} and corresponding expansions by atoms 
without moment conditions if the underlying sequence spaces for the coefficients are essentially the same (up to an additional
summation over $G \in G^j$). Maybe the most distinguished (inhomogeneous unweighted) examples are the spaces
\begin{\eq}   \label{3.37}
F^s_{\infty,q} (\rn), \qquad s\in \real, \quad 0<q \le \infty,
\end{\eq}
with $F^s_{\infty, \infty} (\rn) = B^s_{\infty, \infty} (\rn)$. As for the usual Fourier--analytical definition, related references
and discussions one may consult \cite[Section 1.1.1, pp.\,1--5]{T20}. This will not be repeated here (these spaces do not play any role
in what follows below). We fix the outcome and indicate where the related ingredients for its proof can be found.

\begin{proposition}   \label{P3.3}
Theorem \ref{T3.1} remains valid for the spaces $F^s_{\infty,q} (\rn)$, $s\in \real$, $0<q \le \infty$.
\end{proposition}

\begin{proof}
The case $F^s_{\infty, \infty} (\rn) = B^s_{\infty, \infty} (\rn)$ is already covered by Theorem \ref{T3.1}. Otherwise one can rely on
the observation that
\begin{\eq}   \label{3.38}
F^s_{\infty,q} (\rn) = L^0 \Fs (\rn), \qquad s\in \real, \quad 0<p<\infty, \quad 0<q<\infty,
\end{\eq}
\cite[Proposition 1.18, pp.\,12--13]{T20} are special so--called hybrid spaces $L^r\! \Fs (\rn)$. The theory of these spaces has been
developed in \cite{T14}. There one finds wavelet expansions  in \cite[Theorem 3.26, p.\,64]{T14} and atomic representations in
\cite[Theorem 3.33, p.\,67]{T14}. Using \eqref{3.38} it comes out that no moment conditions for the underlying atoms  in the spaces
\begin{\eq}   \label{3.39}
F^s_{\infty,q} (\rn), \qquad 0<q<\infty, \quad s> \sigma^n_q,
\end{\eq}
are needed. Then one argue as in Step 1 of the proof of that above theorem. This proves the proposition for the spaces in \eqref{3.39}.
The duality
\begin{\eq}   \label{3.40}
F^s_{1, q} (\rn)' = F^{-s}_{\infty, q'} (\rn), \qquad s\in \real, \quad 1\le q <\infty, \quad \frac{1}{q} + \frac{1}{q'} =1,
\end{\eq}
according to \cite[(1.24), p.\,5]{T20} (and the references given there) shows that one can extend the desired assertion in the same
way as in Step 2 of the proof of the above theorem to the spaces
\begin{\eq}   \label{3.41}
F^s_{\infty,q} (\rn), \qquad s<0, \quad 1<q<\infty.
\end{\eq}
Finally one applies the interpolation \eqref{3.31}, \eqref{3.33} with $p_0 = p_1 = p = \infty$ in the same way as there.
\end{proof}

\begin{remark}  \label{R3.4}
There are further examples. In particular, Step 1 of the proof of the above theorem and also \eqref{3.36} can be extended to some
hybrid spaces $\LA (\rn)$ based on the above references. Duality does not work. But if $\delta <1$ (excluding the exotic case) then 
one can rely on \eqref{3.35} and the counterpart of the lifts in \eqref{3.5} for the spaces $\LA (\rn)$ according to \cite[Theorem
3.72, p.\,102]{T14}. It came out quite recently that it is adequate to reformulate the hybrid spaces as
\begin{\eq}   \label{3.42}
\LA (\rn) = \Lambda^{\vr} \As (\rn), \qquad -n <\vr <0, \quad 0<p<\infty, \quad r = \frac{\vr}{p},
\end{\eq}
$0<q \le \infty$, and to collect them together with their Morrey counterparts $\Lambda_{\vr} \Fs (\rn) = \Lambda^{\vr} \Fs (\rn)$ and
\begin{\eq}   \label{3.43}
\Lambda_{\vr} \Bs (\rn) = \big( \Lambda^{\vr} F^{s_1}_{p,q} (\rn), \Lambda^{\vr} F^{s_2}_{p,q} (\rn) \big)_{\theta,q}
\end{\eq}
$-\infty <s_1 <s_2 <\infty$, $s= (1-\theta) s_1 + \theta s_2$, into so--called $\vr$--clans, \cite[(2.46), Definition 2.15, 
(3.14)]{HaT21}. The above comments about mapping properties of pseudodifferential  operators, based on \eqref{3.18}, \eqref{3.25},
\eqref{3.26}, apply to all these spaces without any additional efforts. On the other hand, pseudodifferential operators related to
the above spaces in their more traditional formulations have already been studied in the literature. Corresponding assertions,
discussions and references may be found in \cite[Chapter 5]{YSY10} and \cite[Section 3.7, pp.\,129--131]{Sic12}.
\end{remark}

\section{Fourier operators}    \label{S4}
\subsection{Preliminaries}   \label{S4.1}
First we recall some abstract notation and assertions for linear compact operators in quasi--Banach spaces.

\begin{definition}   \label{D4.1}
Let $T: \ A \hra B$ be a linear and continuous mapping from the quasi--Banach space $A$ into the the quasi--Banach space $B$. Then
the entropy number $e_k (T)$, $k\in \nat$, is the infimum of all $\ve >0$ such that
\begin{\eq}  \label{4.1}
T (U_A) \subset \bigcup^{2^{k-1}}_{j=1} ( b_j + \ve U_B) \quad \text{for some $b_1, \ldots, b_{2^{k-1}} \in B$,}
\end{\eq}
where $U_A = \{ a\in A: \ \|a \, | A \| \le 1 \}$ and $U_B = \{ b\in B: \ \|b \, | B \| \le 1 \}$.
\end{definition}

\begin{remark}   \label{R4.2}
Basic properties and related references may be found in \cite[Section 1.10, pp.\,55--58]{T06}. We only mention that the linear and
continuous mapping $T: \, A \hra B$ is compact if, and only if, $e_k (T) \to 0$ for $k \to \infty$.
\end{remark}

Let $B$ be a complex infinitely--dimensional quasi--Banach space and let $K: \ B\hra B$ be a linear compact operator. Then its spectrum
in the complex plane consists of the origin $0$ and an at most countably infinite numbers of non--zero eigenvalues of finite algebraic
multiplicity which may
accumulate only at the origin. Recall that the algebraic multiplicity of an eigenvalue $\lambda \not= 0$ of $K$ is the dimension of
\begin{\eq}   \label{4.2}
\big\{ b\in B: \ (K- \lambda \, \id )^k b=0 \ \text{for some $k\in \nat$} \big\}.
\end{\eq}
This well--known assertion is also covered by \cite[Theorem, p.\,5]{ET96}. Let $\{ \lambda_k (K) \}$ be the sequence of all non--zero
eigenvalues of $K$, repeated according to algebraic multiplicity and ordered so that
\begin{\eq}   \label{4.3}
|\lambda_1 (K)| \ge |\lambda_2 (K)| \ge \ldots \ge 0.
\end{\eq}
If $K$ has only $m <\infty$ distinct non--zero
eigenvalues and if $M$ is the sum of their algebraic multiplicities we put $\lambda_j (K) =0$ for
$j>M$. Let $e_k (K)$ be the entropy numbers of $K$. Then
\begin{\eq}   \label{4.4}
| \lambda_k (K)| \le \sqrt{2} \, e_k (K), \qquad k\in \nat,
\end{\eq}
is Carl's inequality proved in Banach spaces in \cite{Carl81} in a larger context. The alternative proof in
\cite{CaT80} was extended in \cite[Section 1.3.4, pp.\,18--22]{ET96} to quasi--Banach spaces. Further details  may be found in 
\cite[Theorem 6.25, p.\,197, Notes 6.77, p.\,212]{HT08}.

Let again $F$ be the Fourier transform in $S'(\rn)$ as introduced in \eqref{2.2}. We dealt in \cite{T21} with mapping properties of 
$F$ in the distinguished function spaces
\begin{\eq}   \label{4.5}
B^s_p (\rn) =B^s_{p,p} (\rn), \qquad s\in \real, \quad 0<p \le \infty,
\end{\eq}
as already mentioned in \eqref{2.10}, \eqref{2.11}. We repeat a few assertions on which we rely below. Let
\begin{\eq}   \label{4.6}
d^n_p = 2n \big( \frac{1}{p} - \frac{1}{2} \big), \qquad n\in \nat, \quad 0<p \le \infty.
\end{\eq}

\begin{proposition}   \label{P4.3}
{\em (i)} Let $1<p \le 2$, $s_1 > d^n_p$ and $s_2 <0$. Then
\begin{\eq}   \label{4.7}
F: \quad B^{s_1}_p (\rn) \hra  B^{s_2}_p (\rn)
\end{\eq}
is compact and
\begin{\eq}    \label{4.8}
e_k (F) \le c
\begin{cases}
k^{\frac{s_2}{n}} &\text{if $s_2 > d^n_p -s_1$}, \\
\big( \frac{k}{\log k} \big)^{\frac{s_2}{n}} (\log k)^{\frac{1}{p} - \frac{1}{2}} &\text{if $s_2 = d^n_p -s_1$}, \cm
k^{- \frac{s_1}{n} + 2(\frac{1}{p} - \frac{1}{2})} &\text{if $s_2 < d^n_p -s_1$},
\end{cases}
\end{\eq}
for some $c>0$ and $2 \le k \in \nat$.
\cm
{\em (ii)} Let $2 \le p<\infty$, $s_1 >0$ and $s_2 <d^n_p$. Then
\begin{\eq}  \label{4.9}
F: \quad B^{s_1}_p (\rn) \hra B^{s_2}_p (\rn)
\end{\eq}
is compact and
\begin{\eq}    \label{4.10}
e_k (F) \le c
\begin{cases}
k^{\frac{s_2}{n} - 2(\frac{1}{p} - \frac{1}{2})} &\text{if $s_2 > d^n_p -s_1 $}, \\
\big( \frac{k}{\log k} \big)^{-\frac{s_1}{n}} (\log k)^{\frac{1}{2} - \frac{1}{p}} &\text{if $s_2 = d^n_p -s_1 $}, \cm
k^{- \frac{s_1}{n}} &\text{if $s_2 < d^n_p -s_1$},
\end{cases}
\end{\eq}
for some $c>0$ and $2 \le k \in \nat$.
\end{proposition}

\begin{remark}   \label{R4.4}
This coincides with the related parts of \cite[Theorem 4.8]{T21}. In this paper one finds also further results and discussions showing
that the above restrictions for $s_1$ and $s_2$ are natural.
\end{remark}

\subsection{Spectral theory}    \label{S4.2}
We call $F_\tau$,
\begin{\eq}   \label{4.11}
(F_\tau f)(x) = \int_{\rn} e^{-ix \xi} \tau (x,\xi) f (\xi) \, \di \xi, \qquad x\in \rn,
\end{\eq}
a {\em Fourier operator} of the class $\Phi^\sigma_{1,\delta} (\rn)$ with $\sigma \in \real$ and $0 \le \delta
\le 1$ if the symbol $\tau (x,\xi) \in C^\infty (\real^{2n} )$ satisfies for some constants $c_{\alpha, \gamma} \ge 0$,
\begin{\eq}   \label{4.12}
\big| D^\alpha_x D^\gamma_\xi \tau (x, \xi) \big| \le c_{\alpha, \gamma} (1+ |\xi|)^{\sigma - |\gamma| + \delta |\alpha|}, \qquad
x\in \rn, \quad \xi \in \rn,
\end{\eq}
$\alpha \in \nat^n_0$, $\gamma \in \nat^n_0$. This is the direct counterpart of the class $\Psi^\sigma_{1, \delta}(\rn)$ of 
pseudodifferential operators as recalled in \eqref{3.1}, \eqref{3.2}. We prefer now $T_\tau$ according to \eqref{1.3} compared with
\eqref{3.1}. But this is immaterial and can be compensated if one replaces $\tau (x,\xi)$ by $\tau(x, -\xi)$. In particular with the
Fourier transform $F$ as in Proposition \ref{P4.3} one can decompose $F_\tau \in \Phi^\sigma_{1, \delta} (\rn)$ as
\begin{\eq}   \label{4.13}
F_\tau = T_\tau \circ F, \qquad T_\tau \in \Psi^\sigma_{1, \delta} (\rn).
\end{\eq}
This gives the possibility to combine Proposition \ref{P4.3} with Theorem \ref{T3.1}. Let again
\begin{\eq}   \label{4.14}
B^s_p (\rn) = B^s_{p,p} (\rn), \qquad s \in \real, \quad 1<p<\infty,
\end{\eq}
and
\begin{\eq}   \label{4.15}
d^n_p = 2n \big( \frac{1}{p} - \frac{1}{2} \big), \qquad 1<p<\infty.
\end{\eq}
As explained in \eqref{4.2}, \eqref{4.3} the non--zero eigenvalues $\lambda_k (K)$ of compact operators $K$ in Banach spaces are
counted with respect to their algebraic multiplicity. 

\begin{theorem}   \label{T4.5}
Let $n\in \nat$, $0 \le \delta \le 1$, $\sigma >0$ and $F_\tau \in \Phi^{-\sigma}_{1, \delta} (\rn)$.
\cm
{\em (i)} Let $1<p \le 2$ and $\sigma >s> d^n_p$. Then 
\begin{\eq}  \label{4.16}
F_\tau: \quad B^s_p (\rn) \hra B^s_p (\rn)
\end{\eq}
is compact and
\begin{\eq}   \label{4.17}
|\lambda_k (F_\tau )|  \le c\,
\begin{cases}
k^{-\frac{1}{n} \min (\sigma -s, s -d^n_p)} &\text{if $ 2s \not= \sigma + d^n_p$}, \cm
\big( \frac{k}{\log k} \big)^{-\frac{\sigma-s}{n}} \, (\log k )^{\frac{1}{p} - \frac{1}{2}} &\text{if $2s = \sigma + d^n_p$},
\end{cases} 
\end{\eq}
for some $c>0$ and all $2\le k \in \nat$.
\cm
{\em (ii)} Let $2 \le p <\infty$ and $0<s<\sigma + d^n_p$. Then
\begin{\eq}   \label{4.18}
F_\tau : \quad B^s_p (\rn) \hra B^s_p (\rn)
\end{\eq}
is compact and
\begin{\eq}   \label{4.19}
|\lambda_k (F_\tau )|  \le c\,
\begin{cases}
k^{-\frac{1}{n} \min (s, d^n_p -s +\sigma)} &\text{if $ 2s \not= \sigma + d^n_p$}, \cm
\big( \frac{k}{\log k} \big)^{-\frac{s}{n}} \, (\log k )^{\frac{1}{2} - \frac{1}{p}} &\text{if $2s = \sigma + d^n_p$},
\end{cases} 
\end{\eq}
for some $c>0$ and all $2 \le k \in \nat$.
\end{theorem}

\begin{proof}
The lift
\begin{\eq}   \label{4.20}
I_\sigma: \quad B^{s-\sigma}_p (\rn) \hra B^s_p (\rn)
\end{\eq}
in \eqref{3.5}, \eqref{3.6} can also be written as
\begin{\eq}   \label{4.21}
I_\sigma f = \big( \langle \xi \rangle^{-\sigma} f^\vee \big)^\wedge, \qquad \langle \xi \rangle = (1+ |\xi|^2 )^{1/2}, \quad \xi \in
\rn.
\end{\eq}
With $\tau_\sigma (x,\xi) = \tau(x,\xi) \langle \xi \rangle^\sigma$ one has $F_{\tau_\sigma} \in \Phi^0_{1, \delta} (\rn)$ and
\begin{\eq}   \label{4.22}
\begin{aligned}
(F_\tau f)(x) &= \int_{\rn} e^{-ix \xi} \tau_\sigma (x,\xi) \langle \xi \rangle^{-\sigma} f(\xi ) \, \di \xi \\
&= \int_{\rn} e^{-ix \xi} \tau_\sigma (x, \xi) (I_\sigma \wh{f}\, )^\vee (\xi) \, \di \xi
\end{aligned}
\end{\eq}
shows that $F_\tau$ can be decomposed as
\begin{\eq}   \label{4.23}
F_\tau = T_{\tau_\sigma} \circ I_\sigma \circ F, \qquad T_{\tau_\sigma} \in \Psi^0_{1, \delta} (\rn),
\end{\eq}
according to \eqref{4.13}. It follows from $s>0$, $1<p<\infty$ and \eqref{3.36} that one has always
\begin{\eq}   \label{4.24}
T_{\tau_\sigma}: \quad B^s_p (\rn) \hra B^s_p (\rn)
\end{\eq}
without additional assumptions about the dual in the exotic case. Then one has by \eqref{4.20} and Proposition \ref{P4.3} with $s_1
=s$ and $s_2 = s -\sigma$ that $F_\tau$ in \eqref{4.16} is compact and
\begin{\eq}   \label{4.25}
e_k (F_\tau) \le c \, e_k \big( F: \ B^s_p (\rn) \hra B^{s- \sigma}_p (\rn) \big), \qquad k\in \nat.
\end{\eq}
The estimates \eqref{4.17}, \eqref{4.19} follow now from \eqref{4.4} and \eqref{4.8}, \eqref{4.10}.
\end{proof}

\begin{remark}   \label{R4.6}
For $H^s (\rn) = B^s_2 (\rn) = B^s_{2,2} (\rn)$ it follows that
\begin{\eq}  \label{4.26}
F_\tau: \quad H^s (\rn) \hra H^s (\rn), \qquad 0<s<\sigma,
\end{\eq}
is compact and
\begin{\eq}   \label{4.27}
|\lambda_k (F_\tau )|  \le c\,
\begin{cases}
k^{-\frac{1}{n} \min (s, \sigma -s)} &\text{if $ 2s \not= \sigma$}, \cm
\big( \frac{k}{\log k} \big)^{-\frac{s}{n}} &\text{if $2s = \sigma$},
\end{cases} 
\end{\eq}
for some $c>0$ and all $2\le k \in \nat$. In any case, the right--hand sides of \eqref{4.17}, \eqref{4.19} decay most rapidly for
$F_\tau$ in \eqref{4.16} with $1<p<\infty$ if $2s = \sigma + d^n_p$.
\end{remark}

\begin{remark}   \label{R4.7}
There is a dual of $F_\tau$ in \eqref{4.11} at least in the simplest case when the symbol in \eqref{4.12} with $-\sigma$ in place of
$\sigma$, $\sigma >0$, is independent of $x\in \rn$, $\tau (x,\xi) = \tau (\xi)$. Then
\begin{\eq}   \label{4.28}
\big( F'_\tau f \big)(\xi) = \tau (\xi) \wh{f} (\xi), \qquad \xi \in \rn.
\end{\eq}
Based on \eqref{3.28} and the duality theory for compact operators  in Banach spaces one can transfer Theorem \ref{T4.5} from $F_\tau$
to $F'_\tau$. But in this rather peculiar case one can argue much simpler, decomposing $F'_\tau$ as
\begin{\eq}   \label{4.29}
F'_\tau = F \circ I_\tau, \qquad I_\tau f = (\tau \wh{f}\, )^\vee
\end{\eq}
with
\begin{\eq}   \label{4.30}
I_\tau: \quad B^s_p (\rn) \hra B^{s+\sigma}_p (\rn)
\end{\eq}
combined with
\begin{\eq}   \label{4.31}
F: \quad B^{s+\sigma}_p (\rn) \hra B^s_p (\rn)
\end{\eq}
in Proposition \ref{P4.3} and the mapping properties described there. We do not go into the details.
\end{remark}

\subsection{Fourier heat operator}   \label{S4.3}
The Gauss--Weierstrass semi--group $W_t$,
\begin{\eq}   \label{4.32}
W_t f (x) = \big( e^{-t |\xi|^2} \wh{f} \, \big)^\vee, \qquad x\in \rn, \quad t>0,
\end{\eq}
plays a fundamental role in many parts of mathematics, including the theory of the spaces $\As (\rn)$,  linear and non--linear heat
equations, Navier--Stokes equations and diverse types of evolutionary equations. Then it is quite natural to ask what can be said
about related {\em Fourier heat operators} $W^t$,
\begin{\eq}   \label{4.33}
W^t f(x) = \int_{\rn} e^{-ix \xi} e^{-t |\xi|^2} f(\xi) \, \di \xi, \qquad x\in \rn, \quad t>0,
\end{\eq}
in the context of the above theory. We formulate the outcome. Let again $d^n_p = 2n (\frac{1}{p} - \frac{1}{2} )$, $1<p<\infty$, as
in \eqref{4.15}. As explained in \eqref{4.2}, \eqref{4.3} the non--zero eigenvalues $\lambda_k (K)$ of compact operators in Banach
spaces are counted with respect to their algebraic multiplicity.

\begin{theorem}   \label{T4.8}
Let $n\in \nat$.
\cm
{\em (i)} Let $1<p \le 2$ and $s> d^n_p$. Then
\begin{\eq}   \label{4.34}
W^t: \quad B^s_p (\rn) \hra B^s_p (\rn)
\end{\eq}
is compact and
\begin{\eq}   \label{4.35}
|\lambda_k (W^t)| \le c \, t^{-s + n(\frac{1}{p} - \frac{1}{2})} \, \Big( \frac{k}{\log k} \Big)^{- \frac{s-d^n_p}{n}} \big( 
\log k)^{\frac{1}{p} - \frac{1}{2}}
\end{\eq}
for some $c>0$, all $2 \le k \in \nat$ and all $0<t \le 1$.
\cm
{\em (ii)} Let $2\le p <\infty$ and $s>0$. Then
\begin{\eq}   \label{4.36}
W^t: \quad B^s_p (\rn) \hra B^s_p (\rn)
\end{\eq}
is compact and
\begin{\eq}   \label{4.37}
|\lambda_k (W^t)| \le c \, t^{-s + n(\frac{1}{p} - \frac{1}{2})} \, \Big( \frac{k}{\log k} \Big)^{- \frac{s}{n}} \big( 
\log k\big)^{\frac{1}{2}- \frac{1}{p}}
\end{\eq}
for some $c>0$, all $2\le k \in \nat$ and all $0<t \le 1$.
\end{theorem}

\begin{proof}
From \eqref{4.11}, \eqref{4.12} it follows $W^t \in \Phi^{-\sigma}_{1, 0} (\rn)$ for any $\sigma >0$. Then one can apply \eqref{4.17},
\eqref{4.19} with $\sigma= 2s - d^n_p$ as the best possible choice. This justifies \eqref{4.35}, \eqref{4.37} with exception of the
$t$--dependence. Instead of the lifting $I_\sigma$ in \eqref{4.20}, \eqref{4.23} one can rely on the  mapping property
\begin{\eq}   \label{4.38}
t^{d/2} \, \| W_t w \, | A^{s+d}_{p,q} (\rn) \| \le c \, \|w \, | \As (\rn) \|, \qquad 0<t<1,
\end{\eq}
with $W_t$ as in \eqref{4.32} for $A \in \{B,F \}$, $s\in \real$, $d>0$ and $0< p,q \le \infty$ as described in \cite[Theorem 3.35,
p.\,110]{T20}. There one finds also related references and explanations. Here $c>0$ in \eqref{4.38} is independent of $t$ with $0<t
<1$. Then the $t$-dependence in \eqref{4.35}, \eqref{4.37} follows from \eqref{4.38} with $d= \sigma =2s - d^n_p$ in place of
$I_\sigma$.
\end{proof}

\begin{remark}   \label{R4.9}
One can extend the above theorem to the fractional Fourier heat operator $W^t_\alpha$,
\begin{\eq}   \label{4.39}
W^t_\alpha f(x) = \int_{\rn} e^{-ix \xi} e^{-t |\xi|^{2\alpha}} f(\xi) \, \di \xi, \qquad x\in \rn, \quad \alpha>0, 
\end{\eq}
based on the fractional counterparts of \eqref{4.32}, \eqref{4.38}. But we do not go into the details.
\end{remark}


\begin{thebibliography}{iiiiii}

\bibitem[Bou88]{Bou88} G. Bourdaud. Une alg\`{e}bre maximale d'op\'{e}rateurs pseudo--differentiels. Comm. PDE {\bfseries 13} (1988),
1059--1083.

\bibitem[Carl81]{Carl81} B. Carl. Entropy numbers, $s$--numbers and eigenvalue problems. J. Funct. Anal. {\bfseries 41} (1981), 
290--306.

\bibitem[CaT80]{CaT80} B. Carl, H. Triebel. Inequalities between eigenvalues, entropy numbers, and related quantities of compact
operators in Banach spaces. Math. Ann. {\bfseries 251} (1980), 129--133.

\bibitem[ET96]{ET96} D.E. Edmunds, H. Triebel. Function spaces, entropy numbers, differential operators. Cambridge Univ. Press,
Cambridge, 1996.

\bibitem[FrJ90]{FrJ90}  M. Frazier, B. Jawerth. A discrete transform and decompositions of distribution spaces. J. Funct. Anal.
{\bfseries 93} (1990), 34--170.

\bibitem[GrT99]{GrT99} L. Grafakos, R.H. Torres. Pseudodifferential  operators with homogeneous symbols. Michigan Math. J. {\bfseries
46} (1999), 261--269.

\bibitem[HT08]{HT08} D.D. Haroske, H. Triebel. Distributions, Sobolev spaces, elliptic differential operators. European Math. Soc. Publishing House, Z\"{u}rich, 2008.

\bibitem[HaT21]{HaT21} D.D. Haroske, H. Triebel. Morrey smoothness spaces: A new approach. Submitted, arXiv:2110.10609 (2021). 

\bibitem[Hor85]{Hor85} L. H\"{o}rmander. The analysis of linear partial differential operators III. Springer, Berlin, 1985.

\bibitem[Hor88]{Hor88} L. H\"{o}rmander. Pseudo--differential operators of type 1,1. Comm. PDE {\bfseries 13} (1988), 1085--1111.

\bibitem[Hor89]{Hor89} L. H\"{o}rmander. Continuity of pseudo--differential operators of type 1,1. Comm. PDE {\bfseries 14} (1989),
231--243.

\bibitem[Pai83]{Pai83} L. P\"{a}iv\"{a}rinta. Pseudo differential operators in Hardy--Triebel spaces. Z. Anal. Anwendungen {\bfseries
2} (1983), 235--242.

\bibitem[Run85]{Run85} T. Runst. Pseudo--differential operators of the "exotic" class $S^0_{1,1}$ in spaces of Besov and 
Triebel--Lizorkin type. Ann. Global Anal. Geom. {\bfseries 3} (1985), 13--28.

\bibitem[RuS96]{RuS96} T. Runst, W. Sickel. Sobolev spaces of fractional order, Nemytskij operators, and nonlinear partial differential
equations. W. de Gruyter, Berlin, 1996.

\bibitem[Sic12]{Sic12} W. Sickel. Smoothness spaces related to Morrey spaces. A survey. I. Eurasian Math. J. {\bfseries 3} (2012), 
110--149. 

\bibitem[Ste93]{Ste93} E.M. Stein. Harmonic analysis. Princeton Univ. Press, Princeton, 1993.

\bibitem[Tay81]{Tay81} M.E. Taylor. Pseudodifferential operators. Princeton Univ. Press, Princeton, 1981.

\bibitem[Tor90]{Tor90} R.H. Torres. Continuity properties of pseudodifferential operators of type 1,1. Comm. PDE {\bfseries 15} (1990),
1313--1328.

\bibitem[Tor91]{Tor91} R.H. Torres. Boundedness results for operators  with singular kernels on distribution spaces. Memoires Amer.
Math. Soc. {\bfseries 442}. Amer. Math. Soc., Providence, R.I., USA, 1991.

\bibitem[T83]{T83} H. Triebel. Theory of function spaces. Birkh\"{a}user, Basel, 1983. 

\bibitem[T87]{T87} H. Triebel. Pseudo--differential operators in $F^s_{pq}$--spaces. Z. Anal. Anwend. {\bfseries 6} (1987),
143--150.

\bibitem[T92]{T92} H. Triebel. Theory of function spaces II. Birkh\"{a}user, Basel, 1992. 

\bibitem[T06]{T06} H. Triebel. Theory of function spaces III.  Birkh\"{a}user, Basel, 2006. 

\bibitem[T08]{T08} H. Triebel. Function spaces and wavelets on domains. European Math. Soc. Publishing House, Z\"{u}rich, 2008. 

\bibitem[T14]{T14} H. Triebel. Hybrid function spaces, heat and Navier-Stokes equations. European Math. Soc. Publishing House, Z\"{u}rich, 2014.

\bibitem[T20]{T20} H. Triebel. Theory of function spaces IV. Birkh\"{a}user, Basel, 2020. 

\bibitem[T21]{T21} H. Triebel. Mapping properties of Fourier transforms. Submitted, arXiv:2112.04896 (2021). 

\bibitem[YSY10]{YSY10} W. Yuan, W. Sickel, D. Yang. Morrey and Campanato meet Besov, Lizorkin and Triebel. Lecture Notes. Math
{\bfseries 2005}, Springer, Heidelberg, 2010.




\end{thebibliography}
\end{document}